\documentclass{amsart}
\usepackage{amssymb}
\usepackage{amsmath}
\usepackage{amsfonts}
\usepackage{graphicx}
\usepackage{epstopdf}
\newtheorem{theorem}{Theorem}[section]
\newtheorem*{trm}{Theorem}

\newtheorem{lemma}[theorem]{Lemma}

\newtheorem*{defin}{Definition}

\renewcommand{\Re}{\mathop{\mathrm{Re}}\nolimits}

\theoremstyle{definition}
\newtheorem{definition}[theorem]{Definition}

\theoremstyle{remark}
\newtheorem{remark}[theorem]{Remark}
\newtheorem{example}[theorem]{Example}

\numberwithin{equation}{section}

\DeclareMathOperator{\supp}{supp}
\DeclareMathOperator{\const}{const}

\begin{document}

\title[Indicator functions]{Indicator functions with uniformly bounded Fourier sums and large
gaps in the spectrum}

\author{S. V. Kislyakov and P. S. Perstneva}

\thanks{This research was supported by the Russian Science Foundation, grant 18-11-00053}



\address{St. Petersburg Department of the 
V. A. Steklov Math. Institute \\
27 Fontanka, St. Petersburg \\
191023, Russia}

\email{skis@pdmi.ras.ru}

\email{deepbrightblue@gmail.com}

\begin{abstract}
Indicator functions mentioned in the title are constructed on an arbitrary nondiscrete locally compact Abelian group of finite dimension. Moreover, they can be obtained by small perturbation from any indicator function fixed beforehand. In the case of a noncompact group, the term ``Fourier sums'' should be understood as ``partial Fourier integrals''. A certain weighted version of the result is also provided. This version leads to a new Men$'$shov-type correction theorem.
\end{abstract}
\maketitle

\section{Introduction}

When stated in precise terms, any specific form of the vague claim that a function and its Fourier transform cannot be too small simultaneously (the celebrated ``uncertainty principle in harmonic analysis'') often turns into the question about a frontier beyond which this claim becomes false. In the range of problems where smallness is understood as the vanishing on a large set, one such frontier was marked recently by Nazarov and Olevskii (see \cite{first}), who constructed  a set $E$ of finite positive measure on the real line such that the Fourier transform of the indicator function $\chi_E$ has support that is fairly thin at infinity. More specifically, given arbitrary mutually nonintersecting intervals $I_k$ in $\mathbb{R}_+$ whose lengths tend to infinity, the support of $\widehat{\chi_E}$ can be placed in $K\cup\left(\bigcup_k(I_k\cup (-I_k))\right)$ for some compact set $K$. We refer the reader to the same paper \cite{first} for a concise survey of known facts about the two ``countries'' separated by the borderline indicated. 

Shortly after, the first author of the present paper observed (see \cite{second}) that a slight modification of the construction in  \cite{first} makes it possible to turn an arbitrary set $A\subset\mathbb{R}$
of finite positive measure into a set $E$ as above by a small perturbation. It was also shown in \cite{second} that, basically with the same proof, a similar result holds for any nondiscrete locally compact Abelian group.\footnote{Formally, the claim is also true for discrete groups, but this is not interesting: the compact set $K$ mentioned in the description of the result is not controlled, all this is about the behavior of Fourier transforms at infinity.} In fact, the invocation of the idea of ``correcting'' a given indicator function was motivated by the results of \cite{1, 2} and \cite{menshov}. For example, in the last paper, an analog of Men$'$shov's classical correction theorem was proved for an arbitrary locally compact Abelian group of finite dimension, moreover, the spectrum of the corrected function was placed in a ``thin'' set like the above union $\bigcup_k(I_k\cup (-I_k))$ in the case of $\mathbb{R}$.

We remind the reader that, on the circle, Men$'$shov's correction theorem says that any measurable (equivalently, any measurable and bounded) function can be modified on a set of an arbitrarily small measure so as to acquire a uniformly convergent Fourier series. Surely, the analog of this statement for general groups also involves a certain type of uniform convergence for Fourier expansions. For the \textit{indicator} function that emerges after correction, one might only hope for the uniform boundedness of partial Fourier sums or integrals instead of uniform convergence, but even this was not ensured in  \cite{second}, moreover, it was hinted there that the method would unlikely be suitable for that.

However, later, a more careful look at the situation showed that, even within the class of indicator functions, we can still combine ``thin'' spectrum, uniform boundedness of partial Fourier integrals, and the idea of correction \textit{\`a la} Men$'$shov. Again, all this can be done on every (nondiscrete) locally compact Abelian group of finite dimension. The present paper is devoted to the exposition of this and related results. The clever nonlinear construction by Nazarov and Olevskii will again be in the core of the arguments, but here this construction will require a more substantial modification than in \cite{second}. Also, some techniques of the paper \cite{menshov} will be invoked (which, however, are rather standard in similar issues).

The paper is organized as follows. In \S1, after necessary preliminaries, we state the results and comment on them. The final \S2 is devoted to the proofs.
 
\section{Preliminaries and precise statements}

Throughout, $G$ will be a nondiscrete locally compact Abelian group of finite dimension, and $\Gamma$ will stand for its group of characters, with Haar measures $dx$ and $d\gamma$; it is assumed that these Haar measures are normalized so that the Fourier transform $\mathcal{F}$, $\mathcal{F} f(\gamma) = \int{f(x)\overline{\gamma(x)}dx}, \gamma \in \Gamma, f \in L^1(G)$, is a unitary operator from $L^2(G)$ onto $L^2(\Gamma)$. We will often write $|e|$ for the Haar measure of a measurable subset $e$ of $G$ or $\Gamma$.

We reproduce a definition from \cite{menshov} (see also \cite{iv2010}).

\begin{defin}[Sufficient pairs]
A pair $(R, S)$ of closed subsets of $\Gamma$ is said to be sufficient if for every compact set $E \subset \Gamma$ there exists a character $\gamma \in \Gamma$ with $-\gamma + E \subset R$ and $\gamma + E \subset S$.
\end{defin}

We shall put the spectrum of corrected (indicator) functions in the union $K\cup R\cup S$, where $(R,S)$ is a sufficient pair in $\Gamma$ and $K$ is a certain compact set depending on the function we are going to modify. The couple $(-\bigcup_k I_k, \bigcup_k I_k)$, which occurred in the Introduction, is sufficient for the (dual) group (of) $\mathbb{R}$. Clearly, a similar construction with intervals replaced by mutually nonintersecting balls of radii tending to infinity provides a sufficient pair in the case of (the dual group of) $\mathbb{R}^n$, and examples for the dual group $\mathbb{Z}^n$ of the torus $\mathbb{T}^n$ are provided in the same way. Moreover, in the case of $\mathbb{R}^n$ or $\mathbb{T}^n$, any sufficient pair includes another one of the above form.

Next, to discuss uniformly bounded Fourier sums (or partial Fourier integrals), we need the notion of a summation basis.

\begin{defin}[Summation bases and the space $u(G, \mathcal{B})$]
For a measurable subset $E$ of $\Gamma$ we define the operator $P_E$ {\rm(}at least on $L^2(G)${\rm)} by the formula
$$P_Ef = \mathcal{F}^{-1}(\chi_E\mathcal{F} f).$$
A subset of $\Gamma$ is said to be bounded if it has compact closure.
Let $\mathcal{B}$ be a family of bounded measurable subsets of $\Gamma$ such that for every compact set $K \subset \Gamma$ there exists $E \in \mathcal{B}$ with $K\subset E$. Such a system $\mathcal{B}$ will be called a summation basis. We define the space $u(G, \mathcal{B})$ to be the set of all  functions  $f\in L^1(G)\cap L^\infty (G)$ for which the norm $\|f\|_u = \sup_{B \in \mathcal{B}}{(\|P_B f\|_\infty + \|f\|_1)}$ is finite.   
\end{defin}

If $G$ is compact, we might drop the $L^1$-norm on the right in the last formula. When we talk about uniformly bounded partial Fourier integrals (or sums), we shall mean the finiteness of the norm
$\|\cdot\|_u$ for a certain fixed summation basis.  Surely, not all summation bases are expected to admit a Men$'$shov-type correction theorem. So, we impose a restriction on them taken from \cite{iv2010} and \cite{2018}.

Let $E$ be a bounded measurable subset of $\Gamma$. We say that a set $B\in\mathcal{B}$ \textit{splits} $E$ if the sets $E \cap B$ and $E \setminus B$ have positive Haar measure. (If $G$ is compact, this simply means that the two sets are nonempty.) Next, we denote by $E_{\mathcal{B}}$ an arbitrary representative of the lowest upper bound (in the complete lattice of measurable sets $\mod 0$) of the collection $\{B \in \mathcal{B}\colon B \; \mbox{splits} \; E\}$ (in symbols, with a slight abuse of notation: $E_{\mathcal{B}} = \cup\{B \in \mathcal{B}\colon B \, \mbox{splits} \, E\}$; if $G$ is compact, then the union in the last formula can be understood literally).

The restriction we are going to impose on a summation basis will depend on a sufficient pair in question.
Here it is.

\begin{defin}[Coordination of a summation basis and a sufficient pair.]
A summation basis $\mathcal{B}$ and a sufficient pair $(R, S)$ are said to be coordinated if for every bounded measurable set $E \subset \Gamma$ the pair $(R \setminus E_{\mathcal{B}}, S \setminus E_{\mathcal{B}})$ is also sufficient. 
\end{defin}

We give here two simple but important examples of summation bases coordinated with any sufficient pair. See \cite{iv2010, 2018} for more information. 

\begin{example}
Let $G$ be compact, infinite, and metrizable. Let $\mathcal{B}=\{B_n\}_{n\in\mathbb{N}}$ be an arbitrary strictly monotone increasing sequence of finite subsets of $\Gamma$ whose union is equal to $\Gamma$.
Then this collection is a summation basis coordinated with an arbitrary sufficient pair in $\Gamma$.

Indeed, under the above assumptions all compact sets in $\Gamma$ are finite, and it is easily seen that $E_{\mathcal{B}}$ is finite for every finite $E$. So it suffices to show that, whenever $(S,R)$ is a sufficient pair in $\Gamma$ and $A\subset\Gamma$ is finite, the pair $(S\setminus A, R\setminus A)$ is also sufficient. For this, taking a finite set $D\subset\Gamma$, we fix $\mu\in\Gamma$ (to be specified later), put $D_1=(-\mu+D)\cup (\mu+D)\cup D$ and find $\lambda\in\Gamma$ with $\lambda + D_1\subset S$ and $-\lambda + D_1\subset R$. Then for some choice of $\mu$ either $(\lambda + D)\cup (-\lambda + D)$ or $(\lambda +\mu + D)\cup (-\lambda -\mu + D)$ does not intersect $A$. For, otherwise both $\lambda$ and $\lambda +\mu$ belong to $H=(D-A)\cup (A-D)$, whence $\mu\in H-H$. Since $H$ is finite and $\Gamma$ is infinite, there is $\mu$ for which the last condition is violated.
\end{example}

\begin{example} The second example pertains to the case where $G$ is $\mathbb{R}^n$ or $\mathbb{T}^n$ (accordingly, $\Gamma$ is either $\mathbb{R}^n$ or $\mathbb{Z}^n$). A nonempty  subset $B$ of $\mathbb{R}^n$ or $\mathbb{Z}^n$ is said to be \textit{solid} if $y=(y_1,\ldots,y_n)\in B$ whenever $|y_j|\le |x_j|$ for $j=1,\ldots,n$ and $x=(x_1,\ldots,x_n)\in B$. \textit{We claim that the collection $\mathcal{B}$ of all solid sets constitutes a summation basis coordinated with an arbitrary sufficient pair $(S,R)$}.
  
To explain this, suppose  for definiteness that we work with $\Gamma =\mathbb{Z}^n$ . (The case of $\mathbb{R}^n$ is similar.) Taking a finite set $K\subset\mathbb{Z}^n$, for each $j=1,\ldots,n$ consider the smallest strip $D_j$ of the form $\{x\in\mathbb{Z}^n\colon |x_j|\le d\}$  that includes $K$, and let $D$ be the union of these strips. It is quite easy to realize that $K_{\mathcal{B}}\subset D$. Now, a simple direct inspection shows that $(S\setminus D, R\setminus D)$ is a sufficient pair. The figure illustrates the  case of $n=2$.
\end{example} 
\begin{center}
\includegraphics[scale=0.5]{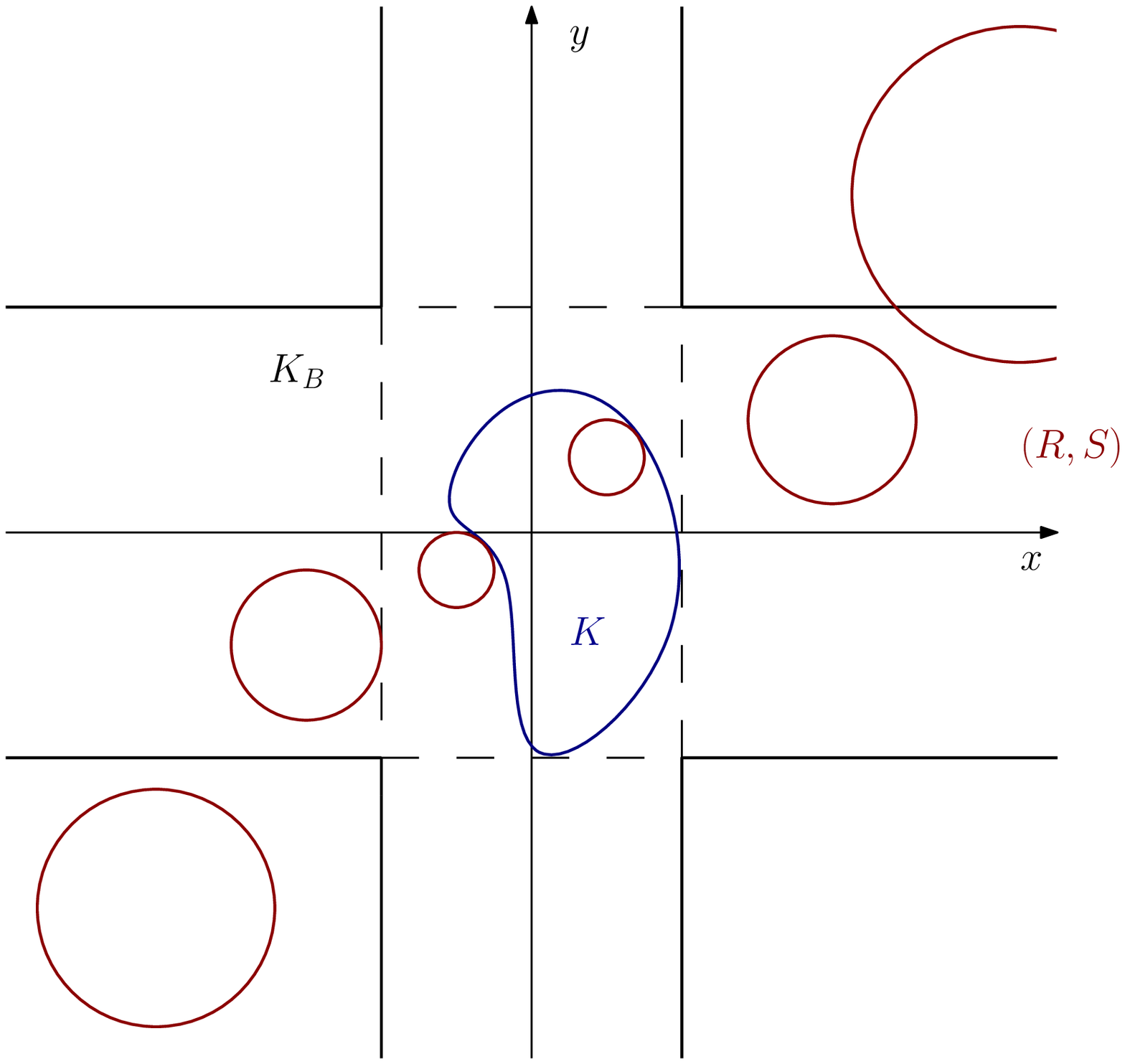}
\end{center}

We pass to precise statements of the results. In fact, we prove a ``weighted'' version of what was discussed in the Introduction. By a weight (maybe the term ``a rail'' would be more appropriate) we mean a uniformly continuous positive function $w$ on $\Gamma$ that is bounded and bounded away from zero. We will modify functions of the form $\chi_a w$  (instead of $\chi_a$) up to functions of the same form. For convenience, we assume that $w\le 1$ (this is merely a normalization condition). Suppose we are given a sufficient pair $(R,S)$ in $G$ and a summation basis $\mathcal{B}$ in $\Gamma$ coordinated with this sufficient pair; the norm $\|\cdot \|_u$ will be related to this summation basis.

\begin{theorem}\label{mainthm}
For  every $\varepsilon > 0$ and an arbitrary measurable subset $a$ of $G$ with $0<|a|<\infty$, there is a measurable subset $b$ of $G$ such that
\begin{enumerate}
\item $\int_{a \triangle b} w^2 <\varepsilon$,
\item the spectrum of $\chi_b w$ is included in $K\cup R\cup S$ for some compact set $K\subset\Gamma$ depending on $a$,
\item the norm $\|w\chi_b\|_u$ is finite.
\end{enumerate}
\end{theorem}

\begin{remark}\label{unity}
The facts discussed in the Introduction follow if $w$ is identically equal to $1$.  In this case it can also be ensured that $|b|=|a|$.
\end{remark}

\begin{remark}\label{estimate}
The claim that $\|w \chi_b\|_u$ is finite can be supplemented with the inequality $\|P_B(w \chi_b)\|\le C$ with $C$ depending only on $\dim G$ whenever $B\in\mathcal{B}$ and $B\supset K$ (this will be verified in the course of the proof).
\end{remark}

Neither the compact set $K$ nor the norms of the $P_B(\chi_b w)$ where $B$ splits $K$ are under control in general, we only know a uniform bound for these norms depending on $a$. However, in some specific cases the estimate can be refined. For example, this is true for the groups $\mathbb{R}$ and $\mathbb{T}$,  the standard summation bases $\{[-N, N], n \in \mathbb{N}\}$ for the circle and $\{[-M, M], M \in \mathbb{R}\}$ for the real line (by the way, these are precisely the bases of solid sets mentioned in Example 1.2) and an arbitrary sufficient pair.

\begin{theorem}\label{line}
Under the assumptions listed in the preceding paragraph, the set $b$ as in Theorem {\rm\ref{mainthm}} can be chosen in such a way that $\|\chi_bw\|_u\le C\log (2+\varepsilon^{-1}\int_aw)$. Here C is a universal constant.
\end{theorem}

\begin{remark}\label{dim0}
The same is true for the dyadic group $D=\{-1,1\}^{\mathbb{N}}$ if we mean a uniform bound for the partial sums of Walsh-Fourier series under the standard enumeration of the Walsh system. Again, a sufficient pair can be taken arbitrarily (note that $1=-1$ in the dual group of $D$, so the notion itself of a sufficient pair simplifies in this case). Moreover, the Walsh system here can be replaced with bounded Vilenkin systems. See Subsection 2.6 for some more information.
\end{remark}

Theorem \ref{line} will allow us to deduce our final result, which is not confined to characteristic functions, and is apparently new. It holds for the groups $\mathbb{R}$ and $\mathbb{T}$ with the standard summation bases and arbitrary sufficient pairs, and also for certain zero-dimensional compact groups; see Subsection 2.6 for the discussion. To a certain extent, this correction theorem absorbs all developments known previously: the modified function has both thin spectrum and bounded Fourier integrals, and obeys a sharp estimate like in Theorem {\rm\ref{line}}. We give the statement for the group $\mathbb{R}$ for definiteness (and with a slightly weaker inequality for the norm $\|\cdot\|_u$ than in Subsection 2.6).

\begin{theorem}\label{fin}
Let $\varepsilon > 0$. Given a function $h\in L^{\infty}(\mathbb{R})$ supported on a set of finite measure and with $\|h\|_{\infty}\le 1$, there is a function $f$ such that $\|f\|_\infty\le 20$, $|\{h\ne f\}|\le\varepsilon$, and $\|f\|_u < C\log (2+\varepsilon^{-1}|\supp h|)$. Furthermore, the spectrum of $f$ is included in $K\cup R\cup S$, where $K$ is a compact set depending on $h$.
\end{theorem}


\section{Proofs}
\subsection{Approximate identities}
We need a certain analog of the Fej{\'e}r kernels for our group $G$. Let $U$  be a compact symmetric neighborhood of zero in the dual group $\Gamma$.  We put $\psi_U = (|U|^{-1/2}\chi_{U}) * (|U|^{-1/2}\chi_{U})$. This is a continuous function on $\Gamma$ with values in $[0, 1]$, supported on the compact set $K = U + U$, and satisfying $\psi_{U}(0) = 1$. Define $\Phi_{U} = \mathcal{F}^{-1}\psi_{U}$; then $\|\Phi_{U}\|_1 = 1$ and $\Phi_{U}\ge 0$. In this subsection, the assumption that $\dim G<\infty$ is not required.

\begin{lemma}\label{Fe}
For a certain family of neighborhoods $U$ of zero in $\Gamma$, the corresponding functions $\Phi_{U}$ form an approximate identity\footnote{By definition, this means that the operators of convolution with these functions converge pointwise  to the identity on $L^p(G)$, $1\le p<\infty$.}  for $G$. 
\end{lemma}

The claim is standard but not quite straightforward because the invocation of the structure theorem seems to be obligatory for the proof. Next, surprisingly, we have not found precisely this statement in standard handbooks. So, for completeness, we sketch the arguments. By the structure theorem (see, e.g., \cite{book}), $G$ splits in the direct product of $\mathbb{R}^k$ and a group containing an open compact subgroup $G_1$. It is quite easy to see that it suffices to prove the claim separately for $\mathbb{R}^k$ and $G_1$. The group $\mathbb{R}^k$ presents no problems (when $U$ runs through the family of cubes centered at zero, we obtain the family of genuine multiple Fej{\'e}r kernels on $\mathbb{R}^k$). We will see that the case of the compact group $G_1$ reduces to considering similar cubes, this time in $\mathbb{Z}^s$ for some $s$. Denote by $\Gamma_1$ the (discrete) dual of $G_1$. Since the operators of convolution with  $\Phi_{U}$ have norm at most one on $L^1(G_1)$, it suffices, given a finite set $C$ of characters on $G_1$, to find $U$ such that this convolution operator is as close to the identity on $C$ as we wish. This means that we must find a symmetric finite subset $U$ of $\Gamma_1$ containing zero such that the function $\psi_U$ is very close to $1$ on $C$.

Now, let $\Delta$ be the subgroup of $\Gamma_1$ generated by $C$. Since $C$ is finite, $\Delta$ is a direct sum of finitely many cyclic groups, whence $\Delta=\mathbb{Z}^s\oplus\Omega$ for some finite Abelian group $\Omega$. If $s=0$, take $\Omega$ for $U$. Otherwise, take a large cube $Q$ centered at zero in $\mathbb{Z}^s$ and put $U=Q\oplus\Omega$. A short reflection shows that the required property of $\psi_U$ is ensured as soon as the diameter of $Q$ is much greater then the diameter of the projection of $C$ to $\mathbb{Z}^s$, and we are done.

\begin{lemma}\label{focus}
Given a compact neighborhood $V$ of zero in $G$ and $\varepsilon >0$, there exists a compact symmetric neighborhood $U$ of zero in $\Gamma$ such that $\int_{G\setminus V}\Phi_U<\varepsilon$.
\end{lemma}
\begin{proof}
This is standard for approximate identities. Indeed, given $f$ in $L^1(G)$, we can find $U$ with
$\|f-f*\Phi_U\|_{L^1(G)}<\varepsilon$. Now, take a symmetric neighborhood $W$ of zero in $G$ such that $W-W\subset V$ and find such a $U$ for $f=|W|^{-1}\chi_W$. Since now $f$ vanishes outside $W$, we have $\int_{G\setminus W} f*\Phi_U (x) dx <\varepsilon$. After plugging the integral formula for convolution in the expression on the left and changing the order of integration, this becomes
\[
\int_G\Phi_U(t)\frac{|(W+t)\setminus W|}{|W|}dt<\varepsilon.
\]
Now, the fraction under the integral sign is equal to $1$ if $t\notin W-W\subset V$. So, it suffices to restrict integration in the last formula to the complement of $V$.
\end{proof}

The last lemma allows us to prove the following statement, which will be useful in the main construction below. Let $w$ be a weight as in Theorem \ref{mainthm}, i.e., a uniformly continuous positive function on $\Gamma$ that is bounded above by 1 and bounded away from zero.

\begin{lemma}\label{convolution}
For every $\eta >0$, we have eventually $w*\Phi_U\le (1+\eta) w$ for the above approximate identity.
\end{lemma}
\begin{proof} Take $\varepsilon >0$ and find a compact symmetric neighborhood $V$ of zero in $G$ such that $|w(x)-w(y)|\le\varepsilon$ whenever $x-y\in V$. Then take $\Phi_U$ as in Lemma \ref{focus} for this $V$ and $\varepsilon$ and write (denoting by $d$ some positive lower bound for $w$):
\begin{multline*}
\Phi_U*w(x)\le\int\limits_V w(x-y)\Phi_U(y) dy+\int\limits_{G\setminus V}\Phi_U (y)dy \\
\le w(x)+2\varepsilon\le w(x)\left(1+\frac {2\varepsilon}{d}\right)\le w(x)(1+\eta)
\end{multline*}
if $\varepsilon$ is sufficiently small.
\end{proof}

\subsection{Covering neighborhoods}
This is a technical ingredient used in various proofs of Men$'$shov-type correction theorems.
\begin{definition}\label{cover}
A compact neighborhood  $V$ of $0$ in $G$ is said to be \textit{covering} if there exists a family $\{x_i\}_{i \in I}$ of points in $G$ such that $G = \cup_{i \in I}{(x_i + V)}$ and $|(x_i + V) \cap (x_j + V)| = 0,\,\,\, i \neq j$. 
\end{definition}

With a covering neighborhood $V$, we associate the family $\{\alpha_i\}_{i \in I}$,
\begin{equation}\label{partition}
\alpha_i(t) =\frac{\chi_V * \chi_V (t - x_i)}{|V|}, \; t \in G,
\end{equation}
of functions on $G$, where $\{x_i\}$ is the family of points mentioned in Definition \ref{cover}. Observe that $\|\alpha_i\|_\infty = 1$ and $\|\mathcal{F}\alpha_i\|_1 = 1$. The last identity will enable us to use combinations of these functions to provide Fourier expansions with uniformly bounded partial integrals (or sums).

\begin{lemma} \label{lm1}
\begin{enumerate}
\item Let $D$ be a compact subset of $G$, and $J$ a finite subset of $I$ such that $D - V \subset \cup_{i \in J}{(x_i + U)}$. Then $\sum\limits_{i \in J}{\alpha_i} = 1$ on $D$, i.e., $\{\alpha_i\}_{i\in J}$ is a partition of unity on $D$. 
\item There exists a base $\mathcal{V}$ of neighborhoods of zero in $G$ such that every $V \in \mathcal{V}$ is a covering neighborhood and $m(V + V) \le 2^{dim G}m(V)$.
\end{enumerate}
\end{lemma}
\begin{proof}
\begin{enumerate}
\item This is clear (however, see \cite{menshov} or \cite{triangle} for details).
\item This fact is obvious for the groups $\mathbb{R}^n$ and $\mathbb{T}^n$: the role of $\mathcal{V}$ can be played by a certain family of cubes centered at zero. For an arbitrary  group $G$, the claim is deduced form these elementary cases with the help of the structure theorem. See again the above references.  
\end{enumerate}
\end{proof}

\begin{remark}\label{multiplicity}
The construction of a covering neighborhood (see the above hint) shows that the supports of the associated functions $\alpha_i$ form a covering of $G$ whose multiplicity is at most $2^{\dim G}$. Again, in the cases of $\mathbb{R}^n$ and the tori, this is straightforward.
\end{remark}

\subsection{Inductive construction}
Here we present a principal ingredient of the proofs of the main results. As it has already been said, we use the ideas of \cite{first}. However, some complications arise. Besides the fact that now we must ensure also the boundedness of Fourier sums (or partial integrals), a technical difference is that presently we shall need a certain double sequence $\{f_k^{(n)}\}_{n \in\mathbb{Z}_+,\,\, 0\le k \le n}$ of functions on $G$ (instead of a single sequence in \cite{first}). It will turn out eventually that the functions $f_n^{(n)}$ converge as $n\to\infty$, and this limit yields the desired function $\chi_b w$ after multiplication by a constant.  We shall proceed by induction on the upper index $n$ (if we view the required functions at the entries of a triangular matrix, this means that at each step we add an entire new row to this matrix).

So, let $w$ and $a$ be as in Theorem \ref{mainthm}. Fix a small $\varepsilon >0$ and a strictly monotone increasing  sequence $\{t_n\}_{n\ge 0}$, $t_n>1$, whose limit $t$ does not exceed $1+\varepsilon$. Also, fix a sequence of positive numbers $\{\rho_n\}_{n\ge 0}$ with $\sum_{n\ge 0}\sqrt{\rho_n}<\varepsilon$. Below we gather certain properties of the functions $f_k^{(n)}$ that will be ensured by induction. The construction will imply some important supplements to these properties, which we do not indicate now. 
\begin{itemize}
\item[(i)] For every $n\ge 0$, we have
\[
0\le f_k^{(n)}\le t_n w,\quad k=0,\ldots,n.
\]
\item[(ii)] The spectra of all $f_k^{(n)}$ are compact and all these functions belong to $L^1(G)\cap C_0(G)$. By $C_0(G)$ we mean the set of all continuous functions on $G$ tending to $0$ at infinity; $C_0(G)=C(G)$ if $G$ is compact. Consequently, all $f_k^{(n)}$ are square integrable.
\item[(iii)] There exists a compact subset $K$ of $G$ such that all functions $f_k^{(n)}$ have compact spectra included in $K\cup R\cup S$, where $(R,S)$ is the sufficient pair mentioned in Theorem \ref{mainthm}.
\item[(iv)] We have 
\begin{equation}\label{zero}
\|f_0^{(0)}-\chi_a w\|_1<\rho_0.
\end{equation}
Next,
\begin{equation}\label{next}
\|f_k^{(n)}-f_k^{(n-1)}\|_1<\rho_{n}
\end{equation}
for $n\ge 1$ and $k=0,\ldots,n-1$ (this relates all functions in the $(n-1)$st row of the matrix mentioned above with \textit{the first $n$ functions} in the $n$th row).
\end{itemize}

Now, we start the construction with $n=0$. To ensure \eqref{zero}, we put $f_0^{(0)}=(\chi_a w)*\Phi_{U_0}$, where $U_0\subset\Gamma$ is chosen in such a way that $\| (\chi_a w)*\Phi_{U_0} - \chi_a w\|_1\le\rho_0$ (see Lemma \ref{Fe}) and $w*\Phi_{U_0}\le t_0 w$ (see Lemma \ref{convolution}). Since $\chi_a w\le w$, the inequalities in (i) for $n=0$ follow. Next, clearly, $f_0^{(0)}\in L^1(G)\cap C_0 (G)$. Moreover, $\mathcal{F}(f_0^{(0)})$ is supported on the compact set $K=U_0-U_0$; this will be the ``$K$'' mentioned in Theorem \ref{mainthm} and in (iii) above. Next, since $\mathcal{F}(f_0^{(0)})$ is integrable, the norm $\|f_0^{(0)}\|_u$ is finite, though no reasonable control of it is available.

Next, suppose that for some $n> 0$ the functions $f_0^{(n - 1)},\ldots,f_{n-1}^{(n-1)}$ have already been constructed. We are going to construct the required collection with the upper index $n$. 
If $k\le n-1$, we take $f_k^{(n)}=\Phi_{U_n}*f_k^{(n-1)}$, where $U_n$ is chosen so as to ensure \eqref{next} and also the inequality $\Phi_{U_n}*w\le\frac{t_n}{t_{n-1}}w$ (see Lemmas \ref{Fe} and \ref{convolution}). In the sequel, we will impose more restrictions on $U_{n}$, compatible with the above. Surely, the present choice of $U_n$ ensures the estimates in (i) for all $k$ except for $k=n$.

The construction of $f_{n}^{(n)}$ is more tricky. First, we introduce the auxiliary function 
\begin{equation}\label{aux}
g_n = f_{n-1}^{(n-1)}\left(1 - \frac{f_{n-1}^{(n-1)}}{t_{n-1}w}\right). 
\end{equation}
Observe that the spectrum of $g_n$ is a compact subset of $\Gamma$ and $g_n$ is nonnegative by (i).

The subsequent arguments will involve certain objects (sets, functions, coefficients, parameters) depending in fact on $n$. But since now $n$ is fixed, this dependence will not always be reflected in the notation. We shall approximate $g_n$  from below by a function suitable for further constructions. 

For this, observe that $g_n$ is square-integrable, so the quantity $\int_G \min (g_n(x),\delta)^2dx$ tends to zero as $\delta\to +0$. Hence, we can find a (small) $\delta>0$ such that for the compact set $C=\{x\in G\colon g_n(x)\ge\delta\}$ we have $\|(g_n-\delta)\chi_C\|_2 > (9/10)\|g_n\|_2$. Denote $g= (g_n-\delta)\chi_C$, then $g$ is continuous and compactly supported, hence uniformly continuous. Next, clearly, $g(x)+\delta/2< g_n(x)$ in a neighborhood $W$ of $C$ with compact closure. By using Lemma \ref{lm1} (with $C$ in the role of ``$D$''), it is easy to realize that there is a (small) covering neighborhood $V$ in $G$ (among other things, we need that $C-V\subset W$) such that $g$ is approximated uniformly and in $L^2(G)$ within any precision prescribed beforehand by a function of the form
\[
h_n=\sum\limits_{i\in J} c_i\alpha_i,\quad c_i=g(x_i)\ge 0,
\]
(the $\alpha_j$ are given by \eqref{partition}; we also use the notation from Lemma \ref{lm1}). Clearly, $h_n\le g_n$ if $V$ is sufficiently small, and all this can be arranged so as to ensure the inequality
\begin{equation}\label{scoop}
\|h_n\|_2^2\ge\frac 1 2 \|g_n\|_2^2.
\end{equation}
Next, by Remark \ref{multiplicity}, we have
\[
h_n(x)^2\le 2^{\dim G}\sum\limits_{i\in J}(c_i\alpha_i (x))^2,\quad x\in G.
\]
Integrating, we arrive at
\begin{equation}\label{below}
\|g_n\|_2^2\le 2^{\dim G+1}\sum\limits_{i\in J}(c_i)^2\|\alpha_i\|_2^2.
\end{equation}

Recall that the set $J$ in the last sum is finite. For a detail in what follows, it is convenient to assume that $J$ is a segment of positive integers.  Now, we want to replace the functions $\alpha_i$, $i\in J$, by the functions $\beta_i=\Phi_{U_{n}}*\alpha_i$, $i\in J$, with compact spectrum. In addition to the restrictions on $U_{n}$ imposed above, we demand that
\begin{equation}\label{step1}
\|\beta_i\|_2^2=\|\Phi_{U_{n}}*\alpha_i\|_2^2\ge \frac{1}{2}\|\alpha_i\|_2^2.
\end{equation}
(See Lemma \ref{Fe}.) 

Finally, we can define the function $f_{n}^{(n)}$: put
\begin{equation}\label{endpoint}
f_{n}^{(n)}=f_{n-1}^{(n)}+\widetilde{h_n},\,\text{ where }\, \widetilde{h_n}=\Re\sum\limits_{i\in J}{c_i\beta_i\gamma_i}.
\end{equation}
Here the $\gamma_i$ are certain characters of the group $G$ (of course, they depend also on $n$, but we do not reflect this in the notation for short). These characters are introduced to eliminate the interference between the summands in \eqref{endpoint} (and elsewhere), which will provide the desired estimate for $\|\cdot\|_u$. The choice of the $\gamma_i$ is described in the following lemma, whose proof is much similar to the proof of Lemma 1 in \cite{second}, and is based entirely on the definition of a sufficient pair and on the fact that the summation basis and the sufficient pair in question are coordinated. We do not reproduce the arguments here. Note that the lemma is quite transparent for the groups $\mathbb{R}^n$ and $\mathbb{T}^n$ in the role of $G$. Surely, in these cases $\Gamma$ has no elements of order two, so item (2) below can be shortened accordingly.  In any case, the $\gamma_i$ are chosen one after another as $i\in J$ grows (we remind the reader that we have assumed that $J$ is a segment of integers). Recall also that the sufficient pair in question is denoted by $(R,S)$.

\begin{lemma}\label{mainlm}
The characters $\gamma_i$ can be chosen in such a way that 
\begin{enumerate}
\item the support of the Fourier transform of \, $\Re(\beta_i\gamma_i)=\beta_i(\gamma_i+\overline{\gamma_i})/2$ lies in $S \cup R$ and intersects neither the spectra of all functions $f_k^{(j)}$ constructed previously {\rm(}i.e., with $j\le n$ and $k\le n-1${\rm)}, nor the spectra of all functions $\Re\beta_s\gamma_s$ for $1 \le s < i$, nor the ``union\footnote{Not quite: the lowest upper bound in the lattice of measurable subsets of $\Gamma$.}'' of all sets in the summation basis $\mathcal{B}$ that split any of these spectra{\rm;}
\item either  $2\gamma_i = 0$ {\rm(}i.e., $\gamma_i(\cdot)^2 = 1${\rm)} or the $\pm 2\gamma_i$ do not lie in the spectrum of $\beta_i^2$.
\end{enumerate}
\end{lemma}

Now, we verify the inequality in (i) for the function  $f_{n}^{(n)}$. By the definition of the $\beta_i$, we have
\[
\left|\Re\sum\limits_{i\in J}{c_i\beta_i\gamma_i}\right|\le\sum\limits_{i\in J}c_i\alpha_i*\Phi_{U_{n}}=h_n*\Phi_{U_{n}}\le g_n*\Phi_{U_{n}}.
\]
So, by \eqref{endpoint} and the definition of $f_{n-1}^{(n)}$, we have
\[
\Phi_{U_{n}}*(f_{n-1}^{(n-1)}-g_n)\le f_{n}^{(n)}\le\Phi_{U_{n}}*(f_{n-1}^{(n-1)}+g_n).
\]
Finally, we use the inductive hypothesis in (i) and the definition \eqref{aux} to conclude that
$f_{n-1}^{(n-1)}-g_n\ge f_{n-1}^{(n-1)}-f_{n-1}^{(n-1)}\ge 0$ and
\[
f_{n-1}^{(n-1)}+g_n\le f_{n-1}^{(n-1)}+ t_{n-1} w\left(1 - \frac{f_{n-1}^{(n-1)}}{t_{n-1}w}\right)=t_{n-1} w.
\]
The desired result follows because $\Phi_{U_{n}}*(t_{n-1} w)\le t_{n} w$ by the choice of $U_{n}$.

This finishes the induction.

\begin{remark}\label{monotone}
It can easily be arranged that $U_0\subset U_1\subset U_2\subset\ldots$. Then the spectrum of $f_k^{(n)}$ does not change when $n$ varies with $k$ fixed (i.e., within each column of the matrix). In the sequel, we will assume that the $U_n$ have this property.
\end{remark}

\subsection{Proof of all claims except the uniform boundedness of partial Fourier integrals}
In this subsection we show that the sequence $\{t_n^{-1}f_n^{(n)}\}_{n \in \mathbb{Z}_+}$ converges  to a function of the form $\chi_b w$, and we verify all metric and spectral conditions for the limit function, except the finiteness of the norm $\|\cdot\|_u$ for it.
\subsubsection{Convergence}

First, we observe that for every $k$ the limit $F_k = \lim_{j \ge k, j \to \infty}{f_k^{(j)}}$ exists in $L^2(G)$. (These are ``the limits along all columns''.) Indeed, by \eqref{next} and (i), we have $\|f_k^{(j)}-f_k^{(j - 1)}\|_2<c\sqrt{\rho_{j}}$, and the quantities on the right were chosen to constitute a convergent series. 

Next, the functions  $\{F_k\}$ form partial sums of an orthogonal series. Indeed, it can easily be seen by induction that, for each $n$, the spectra of the functions $f_0^{(n)}, f_1^{(n)}-f_0^{(n)},\ldots, f_n^{(n)}-f_{n-1}^{(n)}$ are mutually disjoint (see Lemma \ref{mainlm}), hence, these functions are mutually orthogonal, and the claim follows by the limit passage as $n\to\infty$.

It is also easily seen by induction that
\begin{equation}\label{integrals}
\int\limits_G f_k^{(n)}(x)dx = \int\limits_a w(x)dx
\end{equation}
for all $n\ge 0$ and all $k=0,\ldots,n$. Indeed, this is clear for $k=n=0$ and then for $k=0,\,n=1$, because these two functions are obtained from $\chi_a w$ by convolution with positive functions of unit $L^1$-norm. Hence, the spectrum of $f_0^{(1)}$ includes a (neighborhood of) zero. So, $f_1^{(1)}$ is obtained from $f_0^{(1)}$ by adding a function with zero integral (see again Lemma \ref{mainlm}). This proves \eqref{integrals} for $n=1$. Then we pass to $n=2$ in the same way, etc.

Now, we see that
\[
\int\limits_G (F_k)^2 dx=\lim_{n\to\infty}\int\limits_G (f_k^{(n)})^2 dx\le c\int\limits_a w dx,
\]
hence the functions $F_k$ converge to some function $F$ in $L^2(G)$ as $k\to\infty$. Since
\[
\|F_k - f_k^{(k)}\|_2 \le c \sum\limits_{i > k}\sqrt{\rho_i}\,\text{ and }\,\|F_{k-1} - f_{k-1}^{(k)}\|_2 \le c \sum\limits_{i > k}\sqrt{\rho_i},
\]
we see that the sequences $\{f_k^{(k)}\}$ and $\{f_{k-1}^{(k)}\}$ also tend to $F$ in $L^2(G)$ as $k\to\infty$. Hence, $\|\widetilde{h_k}\|_2=\|f_k^{(k)}-f_{k-1}^{(k)}\|_2\to 0$ as $k\to\infty$. 

But the terms $\Re( c_i\beta_i\gamma_i)$ in the formula for $\widetilde{h_k}$ (see \eqref{endpoint}) are mutually orthogonal by construction (see Lemma \ref{mainlm}), hence
\[
\|\widetilde{h_k}\|_2^2= \sum\limits_{i\in J}\|\Re (c_i\beta_i\gamma_i)\|_2^2.
\]
Next, we observe that
\[
\|\Re(\beta_i\gamma_i)\|_2^2 = \frac{1}{4}\int\limits_G{(\overline{\gamma_i} + \gamma_i)^2 \beta_i^2} = \left\{\begin{matrix}
\int\limits_G{\beta_i^2}, \, 2\gamma_i = 0,\\
\frac 1 2\int\limits_G{\beta_i^2}, \, 2\gamma_i \neq 0
\end{matrix}\right.
\]
(in the second line, we have used the fact that the characters $\pm 2\gamma_i$ are not in the spectrum of $\beta_i^2$ if $\gamma_i$ is not of order $2$, see Lemma \ref{mainlm}). Combining \eqref{step1}, \eqref{below}, and \eqref{scoop}, we see that $g_k\to 0$ in $L^2(G)$. Since some subsequence of $\{f_k^{(k)}\}$ must converge to $F$ a.e., looking at formula \eqref{aux} for $g_k$ we realize that at every $x\in G$, either $F(x)=0$, or $F(x)=tw(x)$ (recall that $t_n\to t$). Hence, $F=t\chi_b w$ for some measurable set $b$. We shall show that this $b$ is the required set.

Observe, by the way, that
\begin{equation}\label{id}
\int\limits_b tw(x)dx=\int\limits_a w(x)dx.
\end{equation}
Indeed, inequality \eqref{next} implies that $f_k^{(n)}\to F_k$ also in $L^1(G)$  as $n\to\infty$, hence 
$\int_G F_k(x) dx=\int_a w(x) dx$ for all $k$ by \eqref{integrals}. Since also all $F_k$ are nonnegative, $F$ is integrable and $\int_G F(x) dx\le\int_a w(x) dx$ by the Fatou lemma. In fact, equality occurs here. Indeed, the construction and the Plancherel theorem show that $\mathcal{F}(F)$
and  $\mathcal{F}(F_0)$ coincide a.e. in some neighborhood of zero in $\Gamma$. Since both functions are continuous, they coincide at $0$, whence the claim.

Since $t>1$, we see that $\int_b w\le\int_a w$.

\subsubsection{Correction}
Here we prove that $a$ and $b$ differ only slightly, as claimed in Theorem \ref{mainthm}. By the above discussion, the functions  $F_0$ and $F-F_0$ are orthogonal, hence $\int_G F_0F dx=\int_G (F_0)^2 dx$. Since $f_0^{(0)}$ is at the distance of at most ($\const\varepsilon$) from $F_0$ in $L^2(G)$, we see that
\[
\int\limits_G f_0^{(0)}F =\int\limits_G (f_0^{(0)})^2 +O(\varepsilon).
\]
Now,
\[
\int\limits_{a\cap b}w^2=\frac 1 t \int\limits_G (\chi_aw)F=
\frac 1 t\left(\int\limits_G(\chi_a w-f_0^{(0)})F+\int\limits_G (f_0^{(0)})^2 +O(\varepsilon)\right).
\]
Clearly, the first integral in parentheses is $O(\varepsilon)$ by \eqref{zero}. For the second integral, we write
\[
\int\limits_G (f_0^{(0)})^2\ge\int\limits_G(\chi_a w)^2-\int\limits_G| (f_0^{(0)})^2 -(\chi_a w)^2|\ge \int\limits_a w^2-A_0\int\limits_G| f_0^{(0)} -\chi_a w|.
\]
The subtrahend in the last expression is again $O(\varepsilon)$. Collecting the estimates, we arrive at
$\int\limits_{a\cap b}w^2\ge \frac 1t \int\limits_a w^2 -A_1\varepsilon\ge \int\limits_a w^2-A_2\varepsilon$ if $t$ has been chosen sufficiently close to $1$. Since $\int_b w\le\int_a w$, we arrive at $\int\limits_{a\Delta b}w^2\le A_3\varepsilon$, as required.

\subsubsection{About Remark {\rm\ref{unity}}}
We have $1*\Phi_U=1$ for every $U$. Hence, the numbers $t_n$ are not required in the case where $w$ is identically equal to $1$: the above arguments work with $t_n=1$ for all $n$ (accordingly, $t=1$). Now the claim of the remark follows from \eqref{id}.

\subsubsection{Spectrum}
Needless to say that condition (2) in Theorem \ref{mainthm} is clear from the construction.

\subsection{Uniform boundedness of partial Fourier integrals}
We remind the reader that by Remark \ref{monotone}, the spectrum of $f_k^{(n)}$ does not change when $n\ge k$ varies with $k$ fixed.

Now, take a set $B$ in the summation basis $\mathcal{B}$ in question and find the minimal $k$ such that  $B$ does not include the spectrum of $f_k^{(k)}$ (here and in the next several lines, all inclusions are understood up to a set of zero measure in $\Gamma$). We shall provide some uniform bound for $|P_B f_k^{(k)}|$, and this will suffice. Indeed, suppose we have ensured an upper bound  $D$ for this function. Since for every $j>k$ the function $f_k^{(j)}$ is obtained form $f_k^{(k)}$ by convolution with a nonnegative function with unit integral and since convolution commutes with $P_B$, we see that $|P_B f_k^{(j)}|\le D$. However, by construction (see Lemma \ref{mainlm}), we have
$P_B f_j^{(j)}=P_B f_k^{(k)}$, and we see that  $|P_B f_j^{(j)}|\le D$ for all $j\ge k$, hence also  $|P_B F|\le D$ in the limit.

The nature of the required uniform estimate depends heavily on whether $k=0$ or $k>0$. If $k=0$ (i.e., $B$ splits the support of $f_0^{(0)}$), then $|P_B f_0^{(0)}|\le \|\mathcal{F}f_0^{(0)}\|_1$. The last quantity does not depend on $B$, as required,  but otherwise it is out of our control, we know only that it is finite because $\mathcal{F}f_0^{(0)}$ is bounded and compactly supported.  But if $k>0$ (which is true for sure if $B\supset K$, as in Remark \ref{estimate}), we see that $B$ includes the support of $f_{k-1}^{(k)}$ by the minimality of $k$. Hence, $P_B f_k^{(k)}=f_{k-1}^{(k)}+P_B \widetilde{h_k}$, see \eqref{endpoint}. Since (it can be arranged that) all functions $f_j^{(n)}$ are uniformly bounded, say, by 2, it suffices to estimate the second summand on the right in the last formula.

Recall that in the expression given for $\widetilde{h_k}$ in \eqref{endpoint} we assumed that $J$ is a segment of integers. The way in which we used the order on $J$ in the construction (see again Lemma \ref{mainlm}) shows that there is a unique $l\in J$ with
\[
P_B\widetilde{h_k}=\sum\limits_{i<l}\Re (c_i\beta_i\gamma_i)+P_B(\Re c_l\beta_l\gamma_l).
\]
Now, we remind the reader that $\beta_i=\Phi_{U_{n}}*\alpha_i$, $i\in J$. Hence, recalling the properties of the functions $\alpha_i$ (see \eqref{partition} and Lemma \ref{lm1}) and the fact that $|c_i|=|g(x_i)|\le 2$, we obtain
\[
\left|\sum\limits_{i<l}\Re c_i\beta_i\gamma_i\right|\le 2\sum\limits_{i<l}\Phi_{U_{n}}*\alpha_i=
 2\Phi_{U_{n}}*\left(\sum\limits_{i<l}\alpha_i\right)\le 2
 \]
 and
\[
|P_B(\Re c_l\beta_l\gamma_l)|\le 2\|\mathcal{F}(\frac{\gamma_l+\bar{\gamma_l}}2 \beta_l)\|_1\le\|\mathcal{F}\alpha_l\|_1\le 1.
\]
Collecting the estimates, we see that we have proved Theorem \ref{mainthm} together with Remark \ref{estimate}.

\subsection{Sharp inequalities}
Here we prove Theorem \ref{line}. We saw in the preceding subsection that only $f_0^{(0)}$ presents an obstruction to what we are going to do, and in order to prove the desired claim we must ensure an appropriate control of the partial Fourier integrals (or sums) for this function. In some specific but important cases, this can be done indeed, with the help of a theorem proved in \cite{carleson} and formulated below.

Let  $X$ be a Banach space whose elements are locally integrable functions on a measure space $(S,\mu)$. In a natural way, the space $L^1_{loc}(\mu)$ is locally convex. Next, we denote by $L^{\infty}_0(\mu)$ the space of all essentially bounded functions supported on a set of finite measure. If $g\in L^{\infty}_0(\mu)$, the formula $\Psi_g(u)=\int\limits_S gu d\mu$ defines a linear functional on $L^1_{loc}(\mu)$, hence on every linear subspace of this space. Suppose that the following two conditions are satisfied. 
\begin{enumerate}
\item[A1.] The natural embedding $X \hookrightarrow L^1_{loc}(\mu)$ is continuous and the unit ball of $X$ is weakly compact in $L^1_{loc}(\mu)$.
\item[A2.] For every $g \in L^\infty_0(\mu)$ we have the weak type estimate
$$m(\{|g| > t\}) \le c\frac{ \|\Phi_g\|_{X^*}}t,$$
where $c$ depends only on $X$.
\end{enumerate}

\begin{theorem}\label{correction}
For every $f\in L^\infty(\mu)\cap L^1(\mu)$ with $\|f\|_\infty \le 1$ and every $\varepsilon >0$, there exists a measurable function $\varphi$ with $ 0 \le \varphi \le 1$ such that $\varphi f\in X$, $\mu(\{\varphi \neq 1\}) \le \varepsilon$, and $\|\varphi f\|_X \le \const \log(2+\varepsilon^{-1}\|f\|_1)$.
\end{theorem}

Now, let $\mathcal{B}$ be the summation basis for the unit circle or the real line consisting of all symmetric intervals centered at zero in the dual group $\mathbb{Z}$ or $\mathbb{R}$. For short, we denote the corresponding spaces ``$u$''  by $u(\mathbb{T})$ and $u(\mathbb{R})$, or even simply by $u$. Surely, we take for $\mu$ the Lebesgue measure on the circle or on the line. Then the two above spaces do satisfy Axiom A2 (A1 being a triviality), but this is quite involved. Indeed, eventually this is based on the Carleson almost everywhere convergence theorem for classical Fourier expansions. See \cite{vin} and \cite{spain} for the proof of A2 in these cases, and also Subsections 2.5 and 2.6 in \cite{carleson} for some explanations.

Thus, given a weight $w$ (we still assume that $w\le1$), a number $\varepsilon >0$, and a measurable set $a$ of finite measure on the line or on the circle, we start as at the beginning of Subsection 2.3, but first we modify $f = \chi_a w$ in accordance with Theorem $\ref{correction}$.  The resulting function $\widetilde{f}=f\varphi$ satisfies $0\le\widetilde{f}\le w$ and
\begin{equation}\label{log}
\|\widetilde{f}\|_u \le \const\log\left(2+\frac{\int_aw}{\varepsilon}\right),
\end{equation}
and also $\|\chi_a w-\widetilde{f}\|_1\le\varepsilon$.  Next, we put $f_0^{(0)} = \widetilde{f} * \varphi_{U_0}$ as before, so as to ensure \eqref{zero} but with $\varepsilon +\rho_0$ instead of $\varepsilon$ on the right. This new function $f_0^{(0)}$ will satisfy \eqref{log} because $\widetilde{f}$ does. Then we proceed as previously. All $O(\varepsilon)$`s in Subsection 2.4.2 will remain $O(\varepsilon)$. We do not enter in further details.

\begin{remark}
By \cite{3}, the conclusion of Theorem \ref{line} is also true for certain zero-dimensional compact groups, specifically, for those linked with bounded Vilenkin systems (in particular, the dyadic group with the Walsh system in the usual ordering fits). Again, the verification of Axiom A2 for the corresponding space of functions with uniformly bounded Fourier sums is based eventually on an analog of the Carleson almost everywhere convergence theorem for Vilenkin systems.
\end{remark}

Finally, we prove the announced correction theorem about  essentially bounded functions with support of finite measure, as opposed to the mere indicator functions. We restate it and sketch the proof for the group $\mathbb{R}$, but it will be clear that similar arguments apply to $\mathbb{T}$ and to the zero-dimensional groups mentioned in the last remark. As above, the space $u(\mathbb{R})$ corresponds to the summation basis of symmetric intervals, and also we are given a sufficient pair $(R,S)$ of subsets of $\mathbb{R}$.



\begin{trm}[{\rm\ref{fin}}]
Let $\varepsilon > 0$. Given a function $h\in L^{\infty}(\mathbb{R})$ supported on a set of finite measure and with $\|h\|_{\infty}\le 1$, there is a function $f\in u(\mathbb{R})$ such that $\|f\|_\infty\le 20$, $|\{h\ne f\}|\le\varepsilon$, and $f$ satisfies an estimate like \eqref{log}. Furthermore, the spectrum of $f$ is included in $K\cup R\cup S$, where $K$ is a compact set depending on $h$.
\end{trm}
\begin{proof}
Here $h$ is, in general, complex-valued, but the claim can be reduced to the case of a positive $h$ if we ensure a smaller constant (say, $3$) in place of 20. So, let $0\le h\le 1$, and let $A$ be the support of $h$. Find a compact set $a\subset A$ with $|A\setminus a|\le\varepsilon$ such that $h$ is continuous on  $a$, and then extend $h|_a$ up to a nonnegative uniformly continuous function $v$ on $\mathbb{R}$ with $v\le 1$. Finally, consider the weights (``rails'') $w_1=v+1$ and $w_2=1$ on $\mathbb{R}$.\footnote{Surely, the fact that now the weight $w_1$ is no longer bounded by $1$ from above does not present an obstruction.}

We apply Theorem \ref{line} to the set $a$ consecutively with the weights $w_1$ and $w_2$, obtaining two sets $b_1$ and $b_2$. The function $\chi_{b_1}w_1-\chi_{b_2}w_2$ does the job.
\end{proof}

Note that we cannot eliminate the uncontrollable set $K$ in this statement because of the sharp estimate \eqref{log}. (Should this be possible, all  $L^1$-functions would have spectrum in $R\cup S$.) To the contrary, in the results of \cite{1}, \cite{2}, and \cite{menshov}, the spectrum of a corrected function always lies in $R\cup S$, but, naturally, the control of the size of this function and its partial Fourier integrals is much weaker.

\end{document}